\documentclass[journal,transmag]{IEEEtran}
\usepackage{hyperref,color,breqn,multirow}
\usepackage[top=0.7in, left=0.65in]{geometry}
\usepackage{color}
\usepackage{graphicx}
\hypersetup{hidelinks}
\usepackage{amssymb,amsthm}

\usepackage{booktabs}
\newtheorem{theorem}{Assertion}
\newtheorem{assumption}[theorem]{Assumption}

\theoremstyle{definition}
\newtheorem{remark}[theorem]{Remark}

\def\B{\mathbf{B}}
\def\H{\mathbf{H}}
\def\J{\mathbf{J}}
\def\wJ{\tilde{\mathbf{J}}}

\def\x{\mathbf{x}}

\newcommand{\bigid}{\operatorname{I}}
\newcommand{\bigE}{\operatorname{E}}
\newcommand{\bigH}{\operatorname{H}}

\def\eps{\varepsilon}

\def\RR{\mathbb{R}}
\usepackage[table]{xcolor}
\usepackage[]{float}
\newcommand{\best}[1]{\cellcolor{green!20}\textbf{#1}}
\newcommand{\worst}[1]{\cellcolor{red!20}\textbf{#1}}

\begin{document}
\title{
Efficient evaluation of forward and inverse \\energy-based magnetic hysteresis operators}
\author{\IEEEauthorblockN{Herbert Egger \IEEEauthorrefmark{1,2},
Felix Engertsberger\IEEEauthorrefmark{2} 
and Andreas Schafelner\IEEEauthorrefmark{2} }
\vspace{5pt}
\IEEEauthorblockA{\IEEEauthorrefmark{1}Johann Radon Institute for Computational and Applied Mathematics, Austrian Academy of Sciences, Linz 4040, Austria}
\IEEEauthorblockA{\IEEEauthorrefmark{2}Institute for Numerical Mathematics, Johannes Kepler University, Linz 4040, Austria}}

\IEEEtitleabstractindextext{%
\begin{abstract}
The energy-based vector hysteresis model of Francois-Lavet et al. establishes an implicit relation between magnetic fields and fluxes via internal magnetic polarizations which are determined by convex but non-smooth minimization problems. The systematic solution of these problems for every material point is a key ingredient for the efficient implementation of the model into standard magnetic field solvers. We propose to approximate the non-smooth terms via regularization which allows to employ standard Newton methods for the evaluation of the local material models while being in control of the error in this approximation. We further derive the inverse of the regularized hysteresis operator which amounts to a regularized version of the inverse hysteresis model. 
The magnetic polarizations in this model are again determined by local minimization problems which here are coupled across the different pinning forces. An efficient algorithm for solving the Newton systems is proposed which allows evaluation of the inverse hysteresis operator at the same cost as the forward model. Numerical tests on standard benchmark problems are presented for illustration of our results.
\end{abstract}

}

\maketitle
\thispagestyle{empty}
\pagestyle{empty}

\section{Introduction}
\IEEEPARstart{M}{agnetic} hysteresis models are essential to increase the accuracy of iron loss calculations and calibrate low-fidelity models for simulation of high-power devices like electric machines and transformers \cite{Dlala2006,Billah2020}.
Popular choices are the Jiles-Atherton model or Preisach operators \cite{Jiles1986,Mayergoyz1988} which, however, lack a rigorous thermodynamic background and access to local power losses during transient simulation. Another popular choice are vector play and stop models \cite{Serpico2004,Leite2005} which are based on local energy-balances in principle, but sacrifice those in favor of efficiency of the implementation. 
Thermodynamically consistent models for magnetic hysteresis were introduced by Berquist \cite{Bergqvist1997,Bergqvist1997a} and extended by Henrotte et al \cite{Henrotte2006}. We here consider the fully-implicit time-incremental version of the energy-based hysteresis model proposed by Francois-Lavet et al. \cite{Lavet2013}. This model is truly vectorial, it satisfies local energy balances, provides a clear definition of magnetic losses, and allows to adaptively tune accuracy. 
An equivalent version of the model was studied in \cite{Prigozhin2016,Kaltenbacher2022}.
Extensions of the model accounting for vanishing rotational losses at saturation, the effect of mechanical stresses, or material anisotropy have been considered in \cite{Sauseng2022,Roppert2025,Upadhaya2020}.

\subsection{Energy-based hysteresis model}
Using a fundamental principle for the decomposition of material relations, we may split the magnetic flux 
\begin{align}
    \B = \mu_0 \H + \J
    \label{eq:mat}
\end{align}
into a part proportional to the local magnetic field intensity $\H$ and a magnetic polarization density $\J$.
The latter describes the collective response of the material. A simple choice is to represent $\J = \sum\nolimits_{k=1}^K\J_k$ into multiple partial polarizations which encode the internal states of the system. The hysteresis model of \cite{Lavet2013} uses minimization problems 
\begin{align}
    &\J_k = \arg\min\nolimits_\J U_k(\J) - \langle \H, \J \rangle + \chi_k |\J - \J_{k,p}| \label{eq:for:1}
\end{align}
to implicitly describe the change in partial polarizations $\J_k$ from their previous value $\J_{k,p}$ in response to an applied magnetic field $\H$. 
The internal energy densities $U_k(\cdot)$ and pinning strengths $\chi_k \ge 0$ represent local material characteristics. The symbol $\langle \cdot, \cdot \rangle$ denotes the Euclidean scalar product and $|x|  = \sqrt{\langle x, x \rangle}$ the induced norm.
The model \eqref{eq:mat}--\eqref{eq:for:1} implicitly defines a relation $\B = \mathcal{B}(\H;\{\J_{k,p}\})$ between magnetic field $\H$ and induction $\B$, which additionally depends on the value of the previous internal states. 
The evaluation of this material law requires the numerical solution of convex but non-smooth minimization problems \eqref{eq:for:1} which can be done in parallel and using a duality approach proposed in \cite{Prigozhin2016}. The extension of this method to three dimensional problems and models using effective fields, however, is not directly possible. 

The mapping $\H \mapsto \mathcal{B}(\H;\{\J_{k,p}\})$ can be shown to be strongly monotone and Lipschitz continuous~\cite{egger2025semi}. 
Although the function is not differentiable in a classical sense, these properties at least partially explain the successful use of the model for the simulation of corresponding field problems based on the magnetic scalar potential \cite{Prigozhin2016,Jacques2015,Domenig2024}; see \cite{egger2025semi,egger2025quasi} for a rigorous analysis and provably convergent methods.

\subsection{Inverse hysteresis operator}
By inverting the relation $\B = \mathcal{B}(\H;\{\J_{k,p}\})$ one can obtain the inverse relation $\H = \mathcal{H}(\B;\{\J_{k,p}\})$ which is better suited for the implementation into vector potential formulations of corresponding magnetic field problems. 
Note that the existence of the inverse hysteresis operator $\mathcal{H}(\B;\{\J_{k,p}\})$ is guaranteed by the monotonicity properties of the forward operator.
The evaluation of the inverse operator by numerical inversion of the forward operator was studied in \cite{Jacques2015}. Since it involves the inversion of an implicitly defined function which moreover is not differentiable, this strategy turns out to be quite cumbersome~\cite{Jacques2018}.
In \cite{egger2024inverse} we showed that the inverse hysteresis operator can actually be described in a very similar manner as the forward operator, i.e., by using the fundamental relation
\begin{align}
\H = \nu_0 (\B - \J)
\end{align}
together with the splitting $\J=\sum\nolimits_k \J_k$ and defining the partial magnetic polarizations implicitly via solution of 
\begin{align} \label{eq:inv:1}
\min_{\{\J_k\}} \frac{\nu_0}{2} |\B - \sum\nolimits_k \J_k|^2 + \sum\nolimits_k U_k(\J_k) + \chi_k |\J_k - \J_{k,p}|.
\end{align}
This again is a convex and non-smooth minimization problem which, in contrast to \eqref{eq:for:1}, now requires to solve for all magnetic polarizations simultaneously. 
As a consequence, the duality approach of \cite{Prigozhin2016} is no longer available for the efficient solution of \eqref{eq:inv:1}. 
As mentioned in \cite{Prigozhin2016}, a similar difficulty also arises in the forward hysteresis model when replacing $\H$ by an effective field $\H_{eff}=\H+\alpha \J$ in the minimization problems \eqref{eq:for:1}.

\subsection{Scope and main contributions}
In this paper, we discuss the efficient numerical evaluation of forward and inverse energy-based hysteresis operators based on the following rationale.  
In a first step, we approximate the terms $|\J_k - \J_{k,p}| \approx |\J_k - \J_{k,p}|_\eps$, $\eps>0$ in \eqref{eq:for:1} respectively \eqref{eq:inv:1} by a regularized version of the norm which is differentiable at $\J_k = \J_{k,p}$. 
This allows to employ standard Newton methods with line search for the solution of the regularized minimization problems. 
We further show that the error introduced by this approximation can be fully controlled by choosing $\eps$ sufficiently small. 
Moreover, we prove that the inverse of the regularized forward hysteresis operator amounts to the regularized version of the inverse hysteresis operator, which can be done by elementary arguments. 
In a final step, we show that the coupling between the partial polarizations in \eqref{eq:inv:1} can be overcome in the numerical solution by the Newton method. As a consequence, the inverse hysteresis operator can be evaluated at essentially the same cost as the forward hysteresis operator. 
The theoretical results are demonstrated by numerical tests for some typical benchmark problems.

\section{Regularized hysteresis operator}
We utilize $|\x|_{\eps} := \sqrt{|\x|^2 + \eps}$, $\eps \ge 0$ as approximation for the Euclidean norm $|\x|$. For $\eps>0$, the regularized norm is infinitely differentiable, while for $\eps=0$, we have $|\x|_0=|\x|$. Moreover, $|\x| \le |\x|_\eps \le |\x| + \sqrt{\eps}$; hence the norm $|\x|$ can be approximated with any desired accuracy by smooth functions.
We consider hysteresis models with $K>0$ partial polarization and make the following assumptions for our analysis.
\begin{assumption}     \label{ass:1}
The internal energy densities are given by $ U_k(\J)= -\frac{2 A_{s,k} J_{s,k}}{\pi} \log(\cos(\frac{\pi}{2}\frac{\J}{J_{s,k}}))$ with given parameters $A_{s,k},J_{s,k} > 0$; furthermore $\chi_k \geq 0$ for $k = 1,\dots,K$.
\end{assumption}
As a consequence, the minimization problems \eqref{eq:for:1} and \eqref{eq:inv:1} are strongly convex and uniquely solvable. 
Let us note that other choices for the energy densities $U_k(\J)$ and anisotropic variants $|\chi_k (\J_k - \J_{k,p})|$ of the pinning terms with matrix valued parameters $\chi_k$ could be considered with similar arguments~\cite{Prigozhin2016}.

\subsection{Regularized problem and regularization error}
Following the construction of the forward hysteresis operator, we define $\B = \mu_0 \H + \sum\nolimits_k \J_k^\eps$ with 
\begin{align}
    &\J_k^\eps = \arg\min_\J U_k(\J) - \langle \H, \J \rangle + \chi_k |\J - \J_{k,p}|_{\eps}.\label{eq:for:reg:1}
\end{align}
Note that the difference to \eqref{eq:for:1} only appears in the last term. Since the regularized norm $|\x|_\eps$ is convex, the existence of unique minimizers again follows immediately. 
Moreover, the error introduced by regularization can be estimated as follows. 
\begin{theorem}\label{theo:1}
Let Assumption \ref{ass:1} hold and $\H,\J_{k,p}$ be given. 
Further let $\J_k$ and $\J_k^\eps$, $\eps>0$ denote the solutions of \eqref{eq:for:1} and \eqref{eq:for:reg:1}, respectively. 
Then $|\J_k - \J_k^\eps| = O(\sqrt{\eps})$. 
\end{theorem}
\begin{proof}
The function $F_\eps(\J) = U_k(\J) - \langle \H,\J\rangle + \chi_k |\J - \J_{k,p}|_\eps$ is infinitely differentiable on its domain and its Hessian satisfies $\nabla^2 F_\eps(\J) \ge \nabla^2 U_k(\J) \ge \gamma I$ with $\gamma>0$ depending only on $A_{s,k}$ and $\J_{s,k}$. This shows in particular that $F_\eps(\J)$ is strongly convex, uniformly for all $\eps \ge 0$.
We can thus estimate
\begin{align*}
\tfrac{\gamma}{2} |\J_k &- \J_k^\eps|^2 
\leq F_\eps(\J_k) - F_\eps(\J_k^\eps) \\
&= [ F_\eps(\J_k) - F_0(\J_k)] + [F_0(\J_k) - F_\eps(\J_k^\eps)] 
\end{align*}
The first term can be approximated by $\sqrt{\epsilon}$, since $|\x|_{\epsilon} \leq |\x| + \sqrt{\epsilon} $, and the second term is non-positive, since $|\x|_0 \le |\x_\eps|$ for $\x=\J_k^\eps - \J_{k,p}$. 
This shows that $|\J_k - \J_k^\eps| \le 2\sqrt{\eps}/\gamma$. 
\end{proof}
The minimizers $\J_k$ of \eqref{eq:for:1} can thus be approximated at any level of accuracy by the minimizers $\J_k^\eps$ of the regularized problems \eqref{eq:for:reg:1}. Moreover, the error in this approximation can be fully controlled by appropriate choice of $\eps>0$.

\subsection{Regularized Newton method}
For the numerical solution of the regularized minimization problems \eqref{eq:for:reg:1}, we apply a damped Newton method.
Starting from an initial guess $\J_k^0$, the further iterates are defined by 
\begin{align} \label{eq:new:1}
\J_k^{n+1} = \J_k^{n} + \tau^n \delta \J_k^n, \qquad n \ge 0
\end{align}
with search directions $\delta \J_k^n$ defined by the Newton systems
\begin{align} \label{eq:new:2}
\nabla^2 F_\eps(\J_k^n) \, \delta \J_k^n = - \nabla F_\eps(\J_k^n).
\end{align}
The function $F_\eps(\cdot)$ is the same as used in the proof of the previous result. 
The step-size $\tau^n$ is determined by \emph{Armijo back-tracking}, i.e.,   
we fix $0<\sigma< 1/2$ resp. $\rho<1$, and define
\begin{align} \label{eq:new:3}
\tau^n = \max\{&\tau= \rho^m, \quad m \ge 0, \quad \text{such that}  \\
F_\eps(\J_k^n &+ \tau \delta \J_k^n) \le F_\eps(\J_k^n) + \tau \sigma \langle \nabla F_{\eps}(\J_k^n), \delta  \J_k^n \rangle \}.    \notag
\end{align}
From standard optimization theory~\cite{Nocedal2006}, we obtain
\begin{theorem}
Let Assumption~\ref{ass:1} hold and $\eps>0$. Then the iteration \eqref{eq:new:1}--\eqref{eq:new:3} converges globally and locally quadratically to the unique minimizer of \eqref{eq:for:reg:1}.      
\end{theorem}

\begin{remark}
Since $\J \in \RR^d$ for dimension $d \le 3$, every step of the Newton iteration requires $O(1)$ operations. Due to the local quadratic convergence, the expected number of Newton iterations is $O(1)$ as well. The numerical effort for evaluating $\B=\mathcal{B}(\H;\{\J_{k,p}\})$ thus is $O(K)$ for $K$ partial polarizations.
\end{remark}

\section{Inverse Hysteresis Operator}
In analogy to before, we define $\H = \nu_0 (\B - \sum_k \J_k^\eps)$ with partial polarizations now defined by 
\begin{align}\label{eq:inv:reg:1}
\min_{\{\J_k\}} \frac{\nu_0}{2} |\B - \sum\nolimits_k \J_k|^2 + \sum\nolimits_k  U_k(\J_k) + \chi_k |\J_k - \J_{k,p}|_{\eps}. 
\end{align}
Existence of a unique set of minimizers $\{\J_k^\eps\}$ again follows from the strong convexity of the objective function. Note that similar to \eqref{eq:inv:1}, but in contrast to \eqref{eq:for:1} and \eqref{eq:for:reg:1}, the minimization has to be done for all partial polarization simultaneously.

\subsection{Analysis of the regularization error}

With a very similar reasoning as employed for the analysis of the forward operator, we obtain the following result. 
\begin{theorem} \label{theo:2}
Let Assumption \ref{ass:1} hold and $\B,\J_{k,p}$, $1,\ldots,K$ be given. 
Further let $\{\J_k\}$ and $\{\J_k^\eps\}$, $\eps>0$ denote the solutions of \eqref{eq:inv:1} and \eqref{eq:inv:reg:1}, respectively. 
Then $\sum_k |\J_k - \J_k^\eps|^2 = O(\sqrt{\eps})$. 
\end{theorem}
\begin{proof}
Let $\wJ=[\J_1;\ldots;\J_K] \in \RR^{dK}$ denote the vector resulting by concatenating the partial polarizations $\J_k$. Then the objective function in \eqref{eq:inv:reg:1} can be written as $G_\eps(\wJ)$, which is strongly convex with parameter $\gamma/2$. To see this, one only needs to employ the strong convexity of the functions $U_k(\cdot)$, and the convexity of the other terms. The rest of the proof is exactly the same as that for Assertion~\ref{theo:1}. 
\end{proof}
We can thus again approximate the evaluation of the inverse hysteresis operator to any level of accuracy by computations of the regularized and smooth approximations.

\subsection{Relation between forward and inverse operator}
The smoothness of the regularized minimization problems \eqref{eq:for:reg:1} and \eqref{eq:inv:reg:1} allows us to establish the following result. 
\begin{theorem} \label{thm:3}
Let Assumption \ref{ass:1} hold and $\H,\J_{k,p}$, $1,\ldots,K$ be given. For $\eps \ge 0$, let $\J_k^\eps$, $k=1,\ldots,K$ denote the unique minimizers of \eqref{eq:for:1}. 
Then $\{\J_k^\eps\}$ are the unique minimizers of \eqref{eq:inv:reg:1} with $\B=\mu_0 \H + \sum_k \J_k^\eps$, i.e., the regularized inverse hysteresis operator is the inverse of the regularized forward operator.
\end{theorem}
\begin{proof}
Let $\eps>0$. Then the minimizers of \eqref{eq:for:reg:1} satisfy
\begin{align*}
\nabla U_k(\J_k^\eps) - \H + \chi_k \frac{\J_k^\eps - \J_{k,p}}{|\J_k^\eps - \J_{k,p}|_\epsilon} = 0 
\end{align*}
for $k=1,\ldots,K$. The optimality conditions for 
\eqref{eq:inv:reg:1}, on the other hand, read
\begin{align}
-\nu_0 (\B - \sum\nolimits_k \J_k) + \nabla U_k(\J_k) + \chi_k \frac{\J_k - \J_{k,p}}{|\J_k - \J_{k,p}|_\epsilon} = 0, 
\end{align}
for $k=1,\ldots,K$. This system is satisfied for $\B=\mu_0 \H + \sum_k \J_k^\eps$ with $\J_k=\J_k^\eps$, which follows from the previous identities. This already shows the claim for $\eps>0$. The result for $\eps=0$ follows by taking the limit $\eps \to 0$ in both problems and using the convergence results from before.
\end{proof}

\subsection{Iterative solution by a damped Newton method}
As noted in the proof of Assertion~\ref{theo:2}, the minimization problem \eqref{eq:inv:reg:1} can be stated compactly as $\min_{\wJ} G_\eps(\wJ)$, where $\wJ=[\J_1;\ldots,\J_K] \in \RR^{dK}$ denotes the vector resulting from concatenation of the partial polarizations $\J_k$. 
Starting from an initial guess $\wJ^0$, we construct the iterates
\begin{align} \label{eq:new:inv:1}
\wJ^{n+1} = \wJ^{n} + \tau^n \delta \wJ^n, \qquad n \ge 0
\end{align}
with search directions $\delta \wJ^n$ defined by the Newton systems
\begin{align} \label{eq:new:inv:2}
\nabla^2 G_\eps(\wJ^n) \, \delta \wJ^n = - \nabla G_\eps(\wJ^n).
\end{align}
The step size $\tau^n$ is again determined by \emph{Armijo back-tracking}
\begin{align} \label{eq:new:inv:3}
\tau^n = \max\big\{&\tau= \rho^m: m \ge 0 \quad \text{such that}  \\
G_\eps(\wJ^n &+ \tau \delta \wJ^n) \le G_\eps(\wJ^n) + \tau \sigma \langle \nabla G_{\eps}(\wJ^n), \delta  \wJ^n\rangle\big\}.    \notag
\end{align}
with appropriate constants $0<\sigma< 1/2$ resp. $\rho<1$. 
From convexity and smoothness of the function $G_\eps(\cdot)$, we immediately obtain the following conclusions.
\begin{theorem}
Let Assumption~\ref{ass:1} hold and $\eps>0$. Then the iteration \eqref{eq:new:inv:1}--\eqref{eq:new:inv:3} converges globally and locally quadratically to the unique minimizer of \eqref{eq:inv:reg:1}.      
\end{theorem}

\begin{remark}
Due to the local quadratic convergence, we expect convergence in $O(1)$ iterations. However, every step of the Newton iteration now requires $O(K^3)$ operations, since $\wJ \in \RR^{Kd}$ and $d \le 3$, and the Hessian $\nabla^2 G_\eps(\wJ)$ is densely populated.  The numerical effort for evaluating the inverse hysteresis operator $\H=\mathcal{H}(\B;\{\J_{k,p}\})$ thus is $O(K^3)$ when using $K$ partial polarizations $\J_k$.
\end{remark}

\subsection{Schur complement}
\label{subsec:schur}
We now study in more detail the computation of the Newton directions \eqref{eq:new:inv:2}, which is the main numerical effort for evaluation of the inverse hysteresis operator. 
For abbreviation, we introduce $g_k(\J_k) := U_k(\J_k)  + \chi_k |\J_k-\J_{k,p}|_{\eps}$ and skip the iteration index $n$. The Newton update $\delta \wJ=[\delta \J_1;\ldots;\delta\J_K]$ is thus determined by the system of equations
\begin{align} \label{eq:newton:inv:1}
\nu_0 \sum\nolimits_\ell  \delta \J_\ell &+  \nabla^2_{\J_k} g_k(\J_k) \,\delta \J_k  \\
&= \nu_0 (\B - \sum\nolimits_\ell \J_k) - \nabla g_k(\J_k) \notag
\end{align}
for $k=1,\ldots,K$, which simply amount to the corresponding blocks of \eqref{eq:new:inv:2}. 
Let us note that the update $\delta \J_k$ for the partial polarizations only couples through the leading linear term. 
By introducing the total update $\delta \J := \sum\nolimits_\ell \delta \J_\ell$ as an auxiliary variable, we may rewrite \eqref{eq:newton:inv:1} equivalently as a block system
\begin{align} 
\begin{pmatrix}
  \nu_0 \bigH_K
  &   \bigE_K^\top \\
  \bigE_K & -\bigid &
\end{pmatrix}
\begin{pmatrix}
  \delta \wJ  \\
  \delta \J 
\end{pmatrix}
=
-\begin{pmatrix}
  \nabla G_\eps(\wJ) \\
  0 
\end{pmatrix}.
\label{eq:newton:inv:2}
\end{align}
Here $\bigH_K^n$ denotes the block diagonal matrix obtained from the partial Hessians $\nabla_{\J_k}^2 g_k^n(\J_k^n)$ and $\bigE_K$ denotes the $K$-fold column-wise concatenation of the $d \times d$ identity matrix $\bigid$. The following result is a key observation for the efficient implementation of the inverse hysteresis operator. 

\begin{theorem}
Let $K>0$ be the number of partial polarizations. Then  the linear system \eqref{eq:newton:inv:2} can be solved in $O(K)$ complexity.
\end{theorem}
\begin{proof}
Let $\tilde b_K = -\nabla G_\eps(\wJ)$ denote the first part of the right hand side in \eqref{eq:newton:inv:2}. We can then express 
\begin{align}
    \delta \wJ = \bigH_K^{-1} (\tilde b_K - E_K^\top \delta \J).
    \label{eq:newton:inv:3}
\end{align}
Plugging this into the second row of \eqref{eq:newton:inv:2} yields a small $d \times d$ linear system of the form
\begin{align*}
[\bigE_K H_K^{-1} \bigE_K^\top + I] \, \delta \J = \bigE_K H_K^{-1} \tilde b_K =: b.
\end{align*}
Since $\bigH_K$ is block-diagonal, the setup of this Schur complement system can be accomplished in $O(K)$ operations, and its solution requires $O(1)$ operations.
Using \eqref{eq:newton:inv:3}, one can finally compute $\delta\wJ$ again in $O(K)$ complexity.
\end{proof}
In summary, we have shown that by using the Schur complement technique, the evaluation of the regularized inverse hysteresis operator $\H=\mathcal{H}(\B;\{\J_{k,p}\}) = \nu_0 (\B - \sum_k \J_k^\eps)$ can be performed in $O(K)$ operations, i.e., with essentially the same optimal complexity as the forward hysteresis operator.

\section{Numerical Illustration}
In the following, we illustrate the theoretical findings of this work by computations for a typical material model and two different magnetic field excitations. We demonstrate the error estimates for the regularized approximations, show the equivalence of the forward and inverse hysteresis models, and compare their computational performance. 

\subsection{Model Problems}
For the internal energy densities $U_k$ in Assumption \ref{ass:1} we use the parameters $A_{s,k} = 50$ and $J_{s,k}$ = $J_s/K$ with $K$ the total number of pinning forces an $J_s = 1.545$. 
For the pinning strengths, we assume a uniform distribution $\chi_k = 140(k-1)/(K-1)$ as it was used in \cite{Prigozhin2016}. Notice that by this choice of the material model we can observe the behavior for an increasing number of pinning forces. For real application purposes, these parameters have to be fitted to real measurements as it was considered for example in \cite{Upadhaya2020,Domenig2024}.
For the numerical experiments we discretize the time-interval $[0,1]$ into $500$ equidistant points $t^i$ and prescribe two different magnetic field excitations $\H_{\text{uni}}^i = 500\cdot [\sin((5/2)\pi \, t^i); 0]$ and a rotational field $\H_{\text{rot}}^i = H_m(t^i)\cdot [\sin(5 \pi t^i);\cos(5 \pi t^i)]$ with $H_m(t^i) = 500\cdot\min(t^i,0.75)$.
We always start in the demagnetized state $\J^\eps_{k,p}(t^0) = 0$ and then solve in every step the problem \eqref{eq:for:reg:1} with $\J^\eps_{k,p} = \J_k^\eps(t^{i-1})$. 
The same procedure is used for the inverse model \eqref{eq:inv:reg:1} with a sequence of magnetic fluxes sequence $\B^i$, $i \ge 1$ used as the input.
\subsection{Solver specification}
For the numerical solution of the regularized minimization problems \eqref{eq:for:reg:1} and \eqref{eq:inv:reg:1}, we utilize a damped Newton method.
For the inverse model \eqref{eq:inv:reg:1}, the straight forward inversion of the Newton system as well as the Schur complement trick are compared.  
In order to guarantee the global convergence of the method, Armijo back-tracking \eqref{eq:new:3} with parameters $\rho = 0.5$ and $\sigma  = 0.1$ is employed. 
The iterations \eqref{eq:for:reg:1} are terminated, when $|\nabla F_\eps(\J^n_k)| \leq \textit{tol}~|\nabla F_\eps(\J^0_k)|$ is reached for the first time. In our simulations, we use \textit{tol} = $10^{-8}$ as a tolerance. The same criterion is used for the inverse iteration \eqref{eq:inv:reg:1}.

\subsection{Regularization error}
In a first series of computations, we validate the error estimates of Assertion~\ref{theo:1} and \ref{theo:2}, where we have proven the convergence of the regularized solution to the solution of the non-smooth problem. 
For solving the unregularized hysteresis model \eqref{eq:for:1}, we use the dual approach proposed by \cite{Prigozhin2016}. The magnetic fluxes $\B^i = \mu_0 \H^i + \sum_k \J_k^i$ are later used as the input for the evaluation of the regularized inverse hysteresis operator. 
As an error measure, we use the relative errors
\begin{align*}
    \frac{\sum\nolimits_i |\B^\eps(t_i) - \B^0(t_i)|^2}{\sum\nolimits_i |\B^0(t_i)|^2} \quad \text{ resp. } \quad \frac{\sum\nolimits_i |\H^\eps(t_i) - \H^0(t_i)|^2}{\sum\nolimits_i |\H^0(t_i)|^2}
\end{align*}
in the Euclidean norm. 
To simulate the typical setting used in practice, we set the tolerance in the Newton solver to $\textit{tol} = \eps$.
\begin{figure}[ht!]
    \centering
    \includegraphics[trim={1cm 0cm 1cm 0.5cm},clip,width=0.45\textwidth]{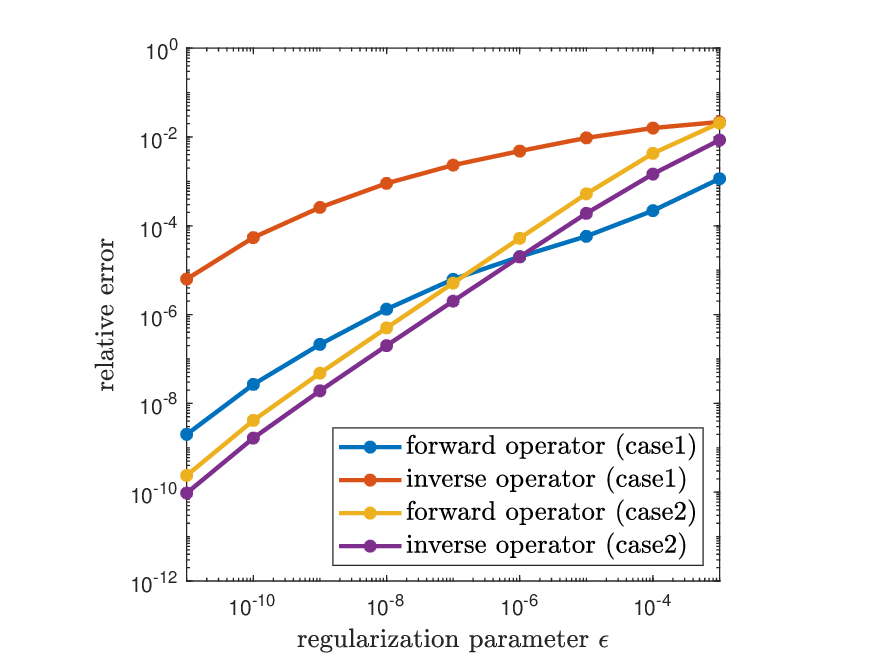}
    \caption{Convergence plot of the relative error for the regularized forward and inverse operator with uniform magnetic field excitation (case 1) and a rotation field (case 2). The number of pinning forces was chosen as $K = 20$.}
    \label{fig:3}
\end{figure}
In Figure \ref{fig:3} the convergence of the regularization errors for the forward and the inverse hysteresis models is visualized for the two different excitation settings. As expected, the regularization errors vanish with $\eps \to 0$. For regularization parameter $\eps = 10^{-6}$ one can already observe excellent agreement with the original model. 
The results of Figure~\ref{fig:3} actually indicate convergence of the form $|\J_k^\eps - \J_k|^2=O(\eps)$ instead of $O(\sqrt(\eps))$. A theoretical justification of this rate might need a different kind of analysis and is left for future investigation.

\subsection{Equivalence of hysteresis operators}
In a second series of tests, we illustrate the result of Assertion~\ref{thm:3}, i.e., that the regularized forward and inverse hysteresis operator are in fact inverse to each other. 
To do so, we take the two different magnetic field excitations $\H_{\text{uni}}^i$ and $\H_{\text{rot}}^i$ as input for the regularized magnetic hysteresis operator, and compute the corresponding sequence of magnetic fluxes $\B^i = \mu_0 \H^i + \sum_k \J_k^\eps(t^i)$. These fluxes then serve as input for the regularized inverse operator.
The corresponding results for the two different excitation sequences are depicted in Fig. \ref{fig:1} and \ref{fig:2}, which validate for the uniform and rotational magnetic field excitation the theoretical result. Here a regularization parameter of $\eps = 10^{-8}$ and a tolerance $\textit{tol}=10^{-8}$ was chosen to stop the Newton iterations for all computations.
\begin{figure}[ht!]
    \centering
    \includegraphics[trim={1cm 0.1cm 1cm 0.6cm},clip,width=0.45\textwidth]{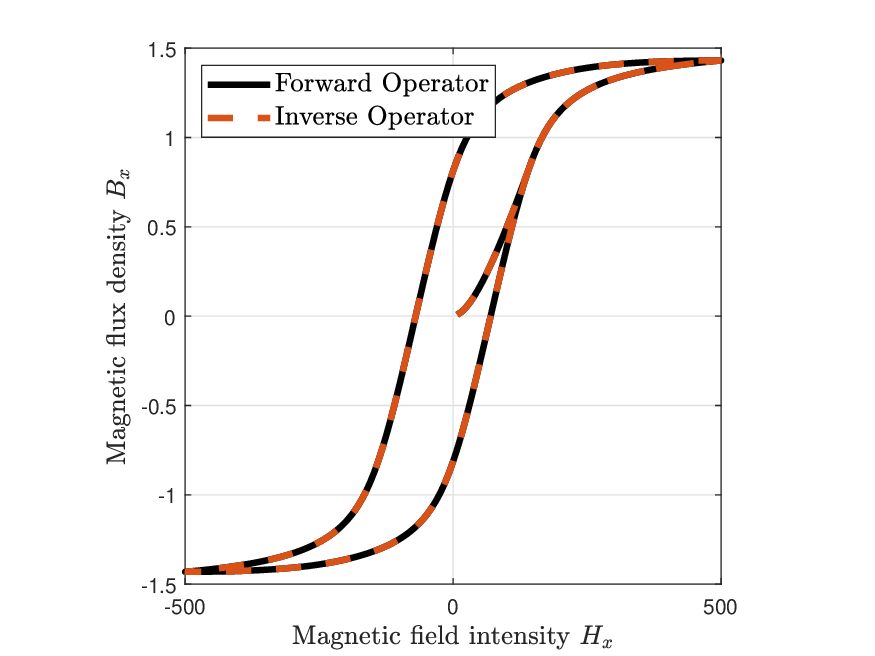}
    \vspace{-0.25cm}
    \caption{Hysteresis loop of the forward and inverse operator with $K = 20$ pinning forces and a unidirectional magnetic field $\H^i = 500\,(\text{sin}(t^i (5/2)\pi),0)$}
    \label{fig:1}
\end{figure}
\begin{figure}[ht!]
    \centering
    \includegraphics[trim={1cm 0.1cm 1cm 0.6cm},clip,width=0.45\textwidth]{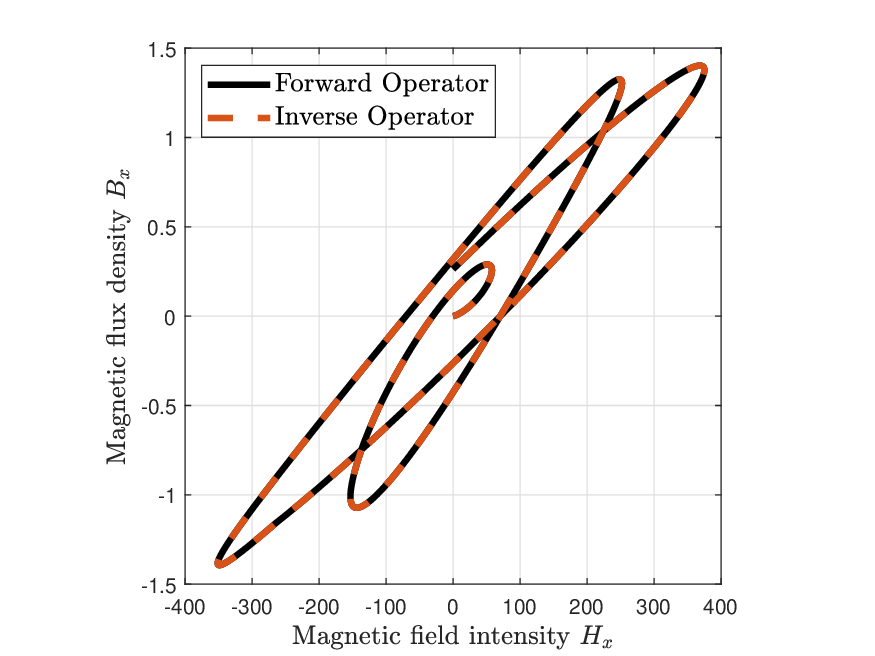}
    \vspace{-0.25cm}
    \caption{Hysteresis loop of the forward and inverse operator with $K = 20$ pinning forces and a rotational magnetic field $\H^i = H_m(t^i)\cdot(\sin(5 \pi t^i),\cos(5 \pi t^i))$ with $H_m(t^i) = 500\cdot\min(t^i,0.75)$}
    \label{fig:2}
\end{figure}

\subsection{Computational performance}
In the last series of tests, we want to compare the computational performance of the algorithms proposed for the evaluation of the regularized forward and inverse hysteresis operators. 
Here we track the wall-time and iteration numbers for the uniform magnetic field excitation $\H_{\text{uni}}$ and a varying number of pinning forces $K$. 
For the forward operator, we also report results for the duality approach of~\cite{Prigozhin2016}, which can be applied for the unregularized problem. 
For the inverse operator, we further compare the straight-forward inversion of the Newton-systems as well as the Schur complement technique discussed in Section \ref{subsec:schur}. 
The corresponding results are depicted in Table \ref{table:1}.

\begin{table}[ht!]
\centering
\caption{Wall-time and iteration numbers for the forward and two implementations of the inverse operator with varying number of pinning forces $K$ and uniform magnetic field excitation $\H_{\text{uni}}$.
Experiments where performed on a laptop with an AMD Ryzen 5 PRO 5675U processor with 4.3 GHz.} 
\label{table:1}
\resizebox{0.49\textwidth}{!}{%
\begin{tabular}{>{\bfseries}l||c c c c c}
\rowcolor{gray!20}
Method $\backslash$ $K$ & $5$ & $10$ & $20$ & $50$ & $100$ \\
\midrule
Prigozhin (Time)                & \best{0ms}   & \best{1ms}   & \best{2ms}   & \best{5ms}  & \best{11ms} \\
Prigozhin (Iterations)          & 0.61         & 0.7          & 0.75         & 0.78        & 0.79         \\
\midrule
Forward (Time)                  & \best{2ms}   & \best{5ms}   & \best{10ms}  & \best{27ms}  & \best{54ms} \\
Forward (Iterations)            & 5.95         & 6.21         & 6.31         & 6.38         & 6.39         \\
\midrule
Inverse, Standard (Time)        & \worst{6ms}  & \worst{14ms} & \worst{40ms} & \worst{289ms} & \worst{1071ms}       \\
Inverse, Standard (Iterations)  & 6.43         & 6.41         & 6.41         & 6.58         & 6.58         \\
\midrule
Inverse, Efficient (Time)       & \best{2ms}   & \best{5ms}   & \best{10ms}  & \best{27ms}  & \best{59ms}        \\
Inverse, Efficient (Iterations) & 6.43         & 6.41         & 6.41         & 6.59         & 6.58         \\
\bottomrule
\end{tabular}}
\end{table}

As expected, the computation times of the regularized forward and inverse operator are very similar when using the Schur complement technique for the latter. The standard implementation of the regularized inverse operator, on the other hand, takes considerably longer for larger numbers $K$ of partial polarizations. The duality-based method of \cite{Prigozhin2016} is still a bit faster than the Newton-type methods used for the realization of the regularized hysteresis operator. This method, however, is less flexible and can not be employed for solution of the regularized problems, for the inverse hysteresis operator, or for the forward model using effective fields. Also the extension to three-dimensional problems is not straight-forward.

\section{Discussion and future Work}
In this work, we investigated the efficient numerical evaluation of forward and inverse energy-based hysteresis operators. By regularization, the non-smoothness of the underlying minimization problems could be eliminated, such that standard Newton-type methods could be used for their reliable and efficient solution. The regularization error was shown to be fully controllable by careful choice of the regularization parameter. In addition, an algebraic trick was presented to accelerate the solution of the Newton-systems for the inverse hysteresis operator. Using such optimized implementations, the forward and inverse hysteresis operator could be evaluated in $O(K)$ complexity, where $K$ denotes the number of partial polarizations used in the models. 
By formal differentiation of the optimality systems governing the minimization problems of the regularized hysteresis operators, one can immediately compute the Jacobians $\nabla_\H \mathcal{B}(\H;\{\J_{k,p}\})$ and $\nabla_\B \mathcal{H}(\B;\{J_{k,p}\})$ of the regularized forward and inverse hysteresis operators. 
This allows the application of Newton-type solvers also for the corresponding magnetic field problems based on the scalar or vector potential formulations. Further investigations in this direction will be reported in future work.

\section*{Acknowledgements}
This work was supported by the joint DFG/FWF Collaborative Research Centre CREATOR (DFG: Project-ID 492661287/TRR 361; FWF: 10.55776/F90).

\bibliographystyle{unsrt}

\end{document}